\documentclass[11pt]{article}
\usepackage{amscd, amsmath, amssymb, amsthm}

\usepackage[all,cmtip]{xy}

\usepackage[pagebackref]{hyperref}

\title{The integral cohomology of the Hilbert scheme
of two points}

\author{Burt Totaro}

\date{  }

\def\Z{\text{\bf Z}}

\def\Q{\text{\bf Q}}

\def\R{\text{\bf R}}

\def\C{\text{\bf C}}

\def\P{\text{\bf P}}

\def\F{\text{\bf F}}

\def\arrow{\rightarrow}

\def\inj{\hookrightarrow}

\def\surj{\twoheadrightarrow}

\def\tors{\text{tors}}
\def\ev{\text{ev}}
\def\odd{\text{odd}}
\def\s{\text{Sq}}
\def\Hom{\text{Hom}}
\def\BM{BM}

\begin{document}
\maketitle

\newtheorem{theorem}{Theorem}[section]

\newtheorem{corollary}[theorem]{Corollary}

\newtheorem{lemma}[theorem]{Lemma}

\theoremstyle{definition}

\newtheorem{definition}[theorem]{Definition}

\newtheorem{example}[theorem]{Example}

\theoremstyle{remark}

\newtheorem{remark}[theorem]{Remark}

For a complex manifold $X$ and a natural number $a$,
the Hilbert scheme $X^{[a]}$
(also called the Douady space) is the space of 0-dimensional
subschemes of degree $a$ in $X$. It is a compactification
of the configuration space $B(X,a)$ of $a$-element subsets of $X$.
The Hilbert scheme
is smooth if and only if $X$ has dimension at most 2
or $a\leq 3$ \cite[equation (0.2.1)]{Cheah}. The integral cohomology
of the Hilbert scheme is more subtle than the rational cohomology.
Markman computed the integral cohomology of the Hilbert schemes
$X^{[a]}$ for $X$ of dimension 2
with effective anticanonical divisor
\cite{Markman}. In this paper, we compute the mod 2
cohomology of $X^{[2]}$ for any
complex manifold $X$, and the integral cohomology of $X^{[2]}$
when $X$ has torsion-free cohomology.

In one way, things are unexpectedly good: the Hilbert scheme $X^{[2]}$
has torsion-free
cohomology if $X$ does (Theorem \ref{torsionfree}).
On the other hand, the details are intricate, and 
it was not clear that complete answers would be possible.
The behavior of the inclusion of the exceptional divisor $E_X$ into
$X^{[2]}$ is related to the Steenrod operations on the mod 2 cohomology
of $X$ (Theorem \ref{kernel}). To explain
one difficulty: some cohomology classes on $X^{[2]}$ can be defined
as the classes of $Y^{[2]}$ in $X^{[2]}$ for complex submanifolds $Y$ of $X$,
which we study in Lemma \ref{complex}.
But because the Hilbert scheme is only defined for complex manifolds,
it is harder to construct ``interesting'' classes on $X^{[2]}$
associated to arbitrary cohomology classes on $X$, for example to
odd-degree cohomology classes. 

Why look at two points? Configurations of two points come up
naturally in geometry, but one especially relevant use of
the Hilbert scheme $X^{[2]}$ is
in Voisin's paper on the universal
$CH_0$ group of cubic hypersurfaces \cite{Voisincubic}.
The background is that major recent advances have been made
in determining which algebraic varieties
are stably rational,
that is, become birational to projective space after multiplying by
projective space of some dimension \cite{Voisindouble,CTP,Totaro}.
These papers are based on the observation that
if a smooth projective variety is stably rational,
then its Chow group of 0-cycles is universally trivial,
meaning that $CH_0$ does not increase when the base field
is increased. 

The Chow group $CH_0\otimes\Q$ is universally
trivial for all rationally connected varieties, and so proving
that varieties of interest are not stably rational
requires looking at torsion
in the Chow group, with the best results
coming from 2-torsion. As a result, Voisin's work
on cubics $X$ uses information on the integral or mod 2
cohomology of the Hilbert scheme $X^{[2]}$, including results
from this paper \cite[proof of Proposition 2.6]{Voisincubic}.
A typical application is that smooth cubic 3-folds in $\C\P^4$
have $CH_0$ universally trivial
for at least a countable union of codimension-3 subvarieties
in the moduli space of cubics \cite[Theorem 1.5]{Voisincubic}.
(Smooth cubic 3-folds are all non-rational by Clemens
and Griffiths \cite{CG}, but it is wide open whether all, or some,
smooth cubic 3-folds are stably rational.)

I thank Claire Voisin for her valuable suggestions.
This work was supported by The Ambrose Monell
Foundation and Friends, via the Institute for Advanced Study,
and by NSF grant DMS-1303105.

\section{Torsion-free cohomology in even degrees}

Here we give a short proof that the Hilbert scheme $X^{[2]}$
of a compact complex manifold has torsion-free cohomology
if the cohomology of $X$ is torsion-free
and concentrated in even degrees. We show this
without the restriction to even degrees and without assuming
compactness in Theorem \ref{torsionfree},
but that proof is considerably harder.

\begin{theorem}
\label{eventorsionfree}
Let $X$ be a compact complex manifold
whose integral cohomology is
torsion-free and concentrated in even degrees.
Then the cohomology of the Hilbert
scheme $X^{[2]}$ is also torsion-free
and concentrated in even degrees.
\end{theorem}

\begin{proof}
Nakaoka and Milgram computed the integral homology
of the symmetric product $S^aX$, the quotient of $X^a$ by
the symmetric group $S_a$,
for any finite CW-complex $X$ and any natural
number $a$ \cite{Milgram}.
We now state their result on $S^2X$
when the homology of $X$ is torsion-free; Theorem \ref{symmprod}
will give their computation of the mod 2 homology of $S^2X$ for any $X$.

\begin{theorem}
\label{symmtorsionfree}
Let $X$ be a finite CW-complex
such that $H_*(X,\Z)$ is torsion-free.
Let $u_0,\ldots,u_s$ be a basis for $H_*(X,\Z)$
as a free graded abelian group. Then $H_*(S^2X,\Z)$
is the direct sum of one copy of $\Z$
in dimension $|u_i|+|u_j|$ for each $0 \leq i \leq j \leq s$
except when $i = j$ and $|u_i|$ is odd,
together with one copy of $\Z/2$ in degrees
$$|u_i|+2, |u_i|+4,\ldots, 2|u_i|-2$$
for each $i$ with $|u_i|$ even and greater than 0,
and one copy of $\Z/2$ in degrees
$$|u_i|+2, |u_i|+4,\ldots,2|u_i|-1$$
for each $i$ with $|u_i|$ odd.
\end{theorem}

Let $X$ be a compact complex manifold
whose integral homology is torsion-free
and concentrated in even degrees. Then Theorem 
\ref{symmtorsionfree} gives that $H_*(S^2X,\Z)$
is also
concentrated in even degrees, although it may have torsion.

A point of the Hilbert scheme $X^{[2]}$ represents either
an unordered pair of distinct points in $X$ or a point $x$ in $X$
together with a complex line in the tangent space $T_xX$. As a result,
the Hilbert scheme $X^{[2]}$
is related to the symmetric square $S^2X$
by a blow-up square:
$$\xymatrix{
E_X \ar[r]\ar[d] & X^{[2]}\ar[d]\\
X\ar[r] & S^2X.
}$$
Here $X \arrow S^2X$ is the diagonal inclusion. For a (real
or complex) vector bundle $V$, we use Grothendieck's
convention that $P(V)$ means
the (real or complex) projective bundle of hyperplanes in $V$,
and $O(1)$ means the quotient line bundle on $P(V)$. Then
the exceptional divisor $E_X$ is the complex
projective bundle $P(T^*X)$ of lines in the tangent bundle $TX$.
(To say that this is a blow-up square means that it is a cartesian
diagram with $X\arrow S^2X$ a closed embedding, $X^{[2]}\arrow
S^2X$ a proper morphism, and $X^{[2]}-E_X\arrow S^2X-X$ an isomorphism.)

The blow-up square gives a long exact sequence of
integral homology groups:
$$H_iE_X \arrow  H_iX\oplus H_iX^{[2]} \arrow H_iS^2X \arrow H_{i-1}E_X.$$
We know that $H_*(X,\Z)$ and $H_*(S^2X,\Z)$
are concentrated in even degrees. Since $E_X$
is the projectivization of a complex vector bundle over $X$,
its homology is also concentrated
in even degrees. So the long exact sequence
gives that $H_*(X^{[2]},\Z)$ is concentrated
in even degrees.

We also want to show that the integral homology of $X^{[2]}$
is torsion-free. Let
$n$ be the complex dimension of $X$.
Because $X^{[2]}$ is a closed oriented real manifold of
dimension $4n$, Poincar\'e duality gives a duality
between the finite abelian groups
$H_i(X^{[2]},\Z)_{\tors}$ and $H_{4n-1-i}(X^{[2]},\Z)_{\tors}$.
Since $H_{\odd}(X^{[2]},\Z)_{\tors} = 0$, it follows that
$H_{\ev}(X^{[2]},\Z)_{\tors} = 0$. Theorem \ref{eventorsionfree} is proved.
\end{proof}

Let $X$ be a compact complex manifold of dimension $n$.
The exceptional divisor
$E_X$ is known to be 2 times an element $e$
in the divisor class group $CH^1X^{[2]}$. This
follows from the existence of the double covering
$g$ from $S := \widetilde{X\times X}$ to $T := X^{[2]}$,
ramified along $E_X$. Namely, we can define $e$
to be $-c_1$ of the line bundle $(g_*O_S)/O_T$.
We also write $e$ for the associated element
of $H^2(X^{[2]},\Z)$ or $H^2(X^{[2]},\F_2)$.

The restriction of $e$ to the exceptional divisor
$E_X=P(T^*X)$ is $e=c_1O(-1)$. The cohomology
of $E_X$ with any coefficient ring is a free module
over $H^*X$ with basis $1,e,\ldots,e^{n-1}$. Let $i\colon E_X\arrow X^{[2]}$
be the inclusion, and let $\pi\colon E_X
\arrow X$ be the projection. To simplify
notation, we omit the symbol $\pi^*$ when considering cohomology
classes on $X$ pulled back to $E_X$.

\section{Main results}

\begin{theorem}
\label{kernel}
Let $X$ be a complex manifold of complex dimension $n$.
Suppose that $X$ is homeomorphic to the complement
of a closed subcomplex in a finite CW-complex; this is no restriction
for $X$ compact.
Then the kernel of the pushforward
homomorphism $i_*\colon H^*(E_X,\F_2)\arrow H^*(X^{[2]},\F_2)$
is spanned over $\F_2$ by the following elements, for $u$ in $H^*(X,\F_2)$:
\begin{align*}
e^j(e^au+e^{a-1}\s^2u+\cdots+\s^{2a}u) & \text{ for }|u|=2a,\; 0\leq j\leq n-1-a;\\
e^j(e^au+e^{a-1}\s^2u+\cdots+\s^{2a}u) & \text{ for }|u|=2a+1,\; 0\leq j\leq n-1-a;\\
e^j(e^{a-1}\s^1u+e^{a-2}\s^3u+\cdots+\s^{2a-1}u) & \text{ for }
|u|=2a,\; 0\leq j\leq n-1-a;\\
e^j(e^a\s^1u+e^{a-1}\s^3u+\cdots+\s^{2a+1}u) & \text{ for }|u|=2a+1,\;
  0\leq j\leq n-2-a.
\end{align*}
\end{theorem}

We have a localization exact sequence, in particular with
$\F_2$ coefficients:
$$\arrow H^{j+1}X^{[2]}\arrow H^{j+1}(S^2X-X)\arrow H^jE_X\arrow H^{j+2}X^{[2]}
\arrow$$
Moreover, the $\F_2$-Betti numbers of $E_X$ and $S^2X-X$ are determined
by those of $X$ (see Theorem \ref{config} for $S^2X-X$). So Theorem
\ref{kernel} determines the $\F_2$-Betti numbers of $X^{[2]}$ in terms
of the action of Steenrod operations on $H^*(X,\F_2)$. The description
is complicated, but this is unavoidable: Example \ref{enriques}
shows that the $\F_2$-Betti numbers of $X^{[2]}$
are not determined by the $\F_2$-Betti numbers of $X$, in general.

On the other hand, the following result implies that
$X^{[2]}$ has several good properties when the integral
cohomology of $X$ has no 2-torsion;
in particular, its $\F_2$-Betti numbers are determined by those of $X$
in that case.

\begin{theorem}
\label{torsionfree}
Let $X$ be a complex manifold of complex dimension $n$
whose integral cohomology has no
2-torsion. Suppose that $X$ is homeomorphic to the complement
of a closed subcomplex in a finite CW-complex; this is no restriction
for $X$ compact.
Then a basis over $\F_2$ for the kernel of the pushforward homomorphism
$i_*\colon H^*(E_X,\F_2)\arrow H^*(X^{[2]},\F_2)$
is given by the elements:
\begin{align*}
e^j(e^au+e^{a-1}\s^2u+\cdots+\s^{2a}u) & \text{ for }|u|=2a,\; 0\leq j\leq n-1-a;\\
e^j(e^au+e^{a-1}\s^2u+\cdots+\s^{2a}u) & \text{ for }|u|=2a+1,\; 0\leq j\leq n-1-a,
\end{align*} 
where $u$ runs through a basis for $H^*(X,\F_2)$.

Moreover, if the integral cohomology of $X$ has no 2-torsion (resp.\ no torsion),
then the integral cohomology of $X^{[2]}$ has no 2-torsion (resp.\ no torsion).
\end{theorem}

The following statement is used in Voisin's paper
on cubic hypersurfaces. It is proved there
in the case of odd-degree complete intersections
in projective space \cite[Lemma 2.8]{Voisincubic}.

\begin{corollary}
Let $X$ be a compact complex manifold whose integral cohomology
has no 2-torsion.
Let $k \geq l$ be integers, and let $\alpha$
be an element of $H^{2k}(E_X,\Z)$
of the form
$$\alpha = e^{k-l}\beta_l + e^{k-l-1}\beta_{l+1} +\cdots$$
with $\beta$ in $H^{2j}(X,\Z)$. If $i_*\alpha$ is divisible by 2
in $H^{2k+2}(X^{[2]},\Z)$ and $2l > k$, then $\beta_l$
is divisible by 2 in $H^{2l}(X,\Z)$.
\end{corollary}

\begin{proof}
Consider $\alpha$ as a class in $H^{2k}(E_X,\F_2)$.
We are assuming that $\alpha$ is in the
kernel of $i_*\colon H^{2k}(E_X,\F_2)\arrow
H^{2k+2}(X^{[2]},\F_2)$. The kernel of $i_*$
on $H^*(E_X,\F_2)$ is
computed in Theorem \ref{torsionfree},
which implies the conclusion here.
\end{proof}

\begin{example}
\label{enriques}
We give an example of compact complex manifolds $X$ and $Y$
with the same $\F_2$-Betti numbers such that $X^{[2]}$
and $Y^{[2]}$ do not have the same $\F_2$-Betti numbers.
First, let $W\arrow \P^1$ be a minimal rational elliptic
surface with section, for example defined by blowing up the intersection
of two cubic curves in $\P^2$. Then $W$ has second Betti number
equal to 10. By Ogg and Shafarevich, for any finite sequence
of integers $m_1,\ldots,m_r\geq 2$, there is a smooth projective
elliptic surface
over $\P^1$ which is a principal homogeneous space for $W\arrow \P^1$
outside $r$ points in $\P^1$ and which has multiple fibers
with multiplicity $m_1,\ldots,m_r$ at those points
\cite[Theorem III.6.12]{FM}. Such a surface automatically has $b_2=10$,
since $b_2(W)=10$
\cite[Lemma I.3.18, Proposition I.3.21, Theorem I.6.7]{FM}.
Let $X$ and $Y$ be such elliptic surfaces
with multiple fibers of multiplicities $2,2$ and $4,4$, respectively.
Then $\pi_1(X)\cong \Z/2$ and $\pi_1(Y)\cong \Z/4$ \cite[Theorem I.2.3]{FM}.
Here $X$ is an Enriques surface and $Y$ has Kodaira dimension 1.

By Poincar\'e duality and the universal coefficient theorem,
the integral cohomology groups of $X$
and $Y$ are:
$$\begin{array}{cccccc}
  & 0 & 1 & 2 & 3 & 4\\
X & \Z & 0 & \Z^{10}\oplus \Z/2 & \Z/2 & \Z\\
Y & \Z & 0 & \Z^{10}\oplus \Z/4 & \Z/4 & \Z
\end{array}$$
It follows that the Enriques surface $X$ and the surface $Y$
have the same $\F_2$-Betti numbers:
1, 1, 12, 1, 1. Because
the Bockstein $\s^1$ is zero on $H^*(Y,\F_2)$ but not on $H^*(X,\F_2)$,
$Y^{[2]}$ has smaller $\F_2$-Betti numbers than $X^{[2]}$. Explicitly,
by Theorem \ref{kernel}, the $\F_2$-Betti numbers are:
$$\begin{array}{cccccccccc}
  & 0 & 1 & 2 & 3 & 4 & 5&6&7&8\\
X^{[2]} & 1 &1&13&15&94&15&13&1&1\\
Y^{[2]} & 1 &1&13&14&92&14&13&1&1
\end{array}$$
\end{example}

\section{The boundary map}

Recall the localization exact sequence with
$\F_2$ coefficients:
$$H^{j+1}X^{[2]}\arrow H^{j+1}(S^2X-X)\arrow H^jE_X\arrow H^{j+2}X^{[2]}.$$
The key step in determining the kernel
of the pushforward $i_*\colon H^jE_X\arrow H^{j+2}X^{[2]}$
is to compute the boundary homomorphism on interesting elements
of $H^{j+1}(S^2X-X)$, as we now do.

\begin{lemma}
\label{boundarymfd}
Let $Z$ be a closed $C^{\infty}$ submanifold of real codimension $r$
in a complex manifold $X$. Let $u$ be the cohomology class of $Z$
in $H^r(X,\F_2)$. Then the boundary in $H^{2r-1}(E_X,\F_2)$
of the class $[S^2Z-Z]$
in $H^{2r}(S^2X-X,\F_2)$ is
$$\begin{cases}
e^{a-1}\s^1u+e^{a-2}\s^3u+\cdots+\s^{2a-1}u&\text{if }r=2a,\\
e^au+e^{a-1}\s^2u+\cdots+\s^{2a}u&\text{if }r=2a+1.
\end{cases}$$
\end{lemma}

\begin{proof}
We can view $S^2Z-Z$ as the interior of a manifold
with boundary, where the boundary is the real
projective bundle $P_{\R}(T^*Z)$ over $Z$. So the boundary
of $[S^2Z-Z]$ is the class $t_*1$ in $H^{2r-1}(E_X,\F_2)$,
where $t$ is the proper map $P_{\R}(T^*Z)\arrow E_X=P_{\C}(T^*X)$,
taking a real line in $T_Z$ at a point $z$ in $Z$
to the complex line that it spans in $T_zX$.

We can factor $t$ as $P_{\R}(T^*Z)\inj P_{\R}(T^*X)|_Z \inj
P_{\R}(T^*X)\surj P_{\C}(T^*X)$. Write $\rho\colon P_{\R}(T^*X)|_Z
\arrow Z$ for the projection. Then
$P_{\R}(T^*Z)$ is the zero set of a transverse section
of the real vector bundle $\Hom(O(-1),\rho^*N_{Z/X})$
over $P_{\R}(T^*X)|_Z$; that section is the one associated
to the subbundle $O(-1) \subset \rho^*TX|_Z$.
So the cohomology class of $P_{\R}(T^*Z)$
on $P_{\R}(T^*X)|_Z$ is the top Stiefel-Whitney class
$w_{r}(O(1)\otimes \rho^*N_{Z/X})$.
The top Stiefel-Whitney class of the tensor product
of a line bundle $L$ with a vector bundle
$W$ of rank $r$ is
$$w_{r}(L\otimes 
W) = (w_1L)^r + (w_1L)^{r-1}w_1W +\cdots+ w_rW.$$
Write $b$ for the class $w_1O(1)$ in $H^1(P_{\R}(TX),\F_2)$.
We deduce that the $\F_2$-cohomology class of $P_{\R}(T^*Z)$ on
$P_{\R}(T^*X)|_Z$ is $b^{r} + b^{r-1}w_1N_{Z/X} +\cdots+w_{r}N_{Z/X}$,
where we omit the symbol $\rho^*$ for cohomology classes
on $Z$ pulled back to $P_{\R}(T^*X)|_Z$.

Write $s$ for the inclusion $Z\inj X$, and $u$ for the cohomology
class $s_*1=[Z]$ in $H^{r}(X,\F_2)$. Then the pushforward
of the class of $P_{\R}(T^*Z)$ from $P_{\R}(T^*X)|_Z$ to $P_{\R}(T^*X)$
is $b^{r}u+b^{r-1}s_*w_1N_{Z/X}+\cdots+s_*w_{r}N_{Z/X}$.
The Steenrod squares of the class $u = [Z]$ in $H^*(X,\F_2)$
are the pushforward to $X$ of the Stiefel-Whitney classes
of the normal bundle $N_{Z/X}$ by the inclusion
$s\colon Z\arrow X$,
$$\s^ju = s_*w_j(N_{Z/X}),$$
by Thom \cite{Thomsteenrod}. So the class of $P_{\R}(T^*Z)$
on $P_{\R}(T^*X)$ is $b^ru+b^{r-1}\s^1u+\cdots+\s^ru$.

Finally, we have to push this class forward via the
$S^1$-bundle $h\colon P_{\R}(T^*X)\arrow P_{\C}(T^*X)$.
(We sometimes write $O(1)_{\R}$ instead of $O(1)$
for the natural real line bundle
on $P_{\R}(T^*X)$, and likewise $O(1)_{\C}$ instead for $O(1)$ for the natural
complex line bundle on $P_{\C}(T^*X)$, to avoid confusion.)
The class $e=c_1O(-1)$ in $H^2(P_{\C}(T^*X),\F_2)$ pulls back
to $b^2$, since the complex line bundle $O(1)_{\C}$ on $P_{\C}(T^*X)$
pulls back to $O(1)_{\R}\otimes_{\R}\C$ on $P_{\R}(T^*X)$,
and $c_1(O(1)_{\R}\otimes_{\R}\C)=w_1(O(1)_{\R})^2=b^2$.
Also, all classes on $P_{\R}(T^*X)$ pulled back from $H^*(X,\F_2)$
(such as the classes $\s^ju$) are also pulled back
from $P_{\C}(T^*X)$. Here $H^*(P_{\R}(T^*X),\F_2)$
is a free module over $H^*(X,\F_2)$ with basis
$1,b,\ldots,b^{2n-1}$, where $n$ is the complex
dimension of $X$. So to compute the pushforward $h_*$
on $\F_2$-cohomology, it suffices to compute $h_*1$
and $h_*b$. Here $h_*1$ is in $H^{-1}(P_{\C}(T^*X),\F_2)=0$,
and so $h_*1=0$. Also, $h_*b$ is in $H^0(P_{\C}(T^*X),\F_2)$,
and so it is either 0 or 1. In fact, $h_*b=1$. This can be proved
by restricting over a point in $X$, and noting
that the inclusion of a hyperplane $b=[\R\P^{2n-2}]$ in
$\R\P^{2n-1}$ composed with the surjection to $\C\P^{n-1}$
has degree $1\pmod{2}$, as it restricts
to a diffeomorphism from $\R^{2n-2}$ to $\C^{n-1}$.

Therefore, for $Z$ of codimension $r=2a$, the boundary
of $[S^2Z-Z]$ in $H^{4a-1}(E_X,\F_2)$ is
$$h_*(b^{2a}u+b^{2a-1}\s^1u+\cdots+\s^{2a}u) =
e^{a-1}\s^1u+e^{a-2}\s^3u+\cdots+\s^{2a-1}u.$$
For $Z$ of codimension $r=2a+1$, the boundary of
$[S^2Z-Z]$ in $H^{4a+1}(E_X,\F_2)$ is
$$h_*(b^{2a+1}u+b^{2a}\s^1u+\cdots+\s^{2a+1}u)=
e^au+e^{a-1}\s^2u+\cdots+\s^{2a}u.$$
\end{proof}

Let $b$ be the element of $H^1(S^2X-X,\F_2)$ associated
to the double cover $X\times X-X\arrow S^2X-X$. 

\begin{lemma}
\label{boundaryb}
Let $Z$ be a closed $C^{\infty}$ submanifold of real codimension $r$
in a complex manifold $X$. Let $u$ be the cohomology class
of $Z$ in $H^r(X,\F_2)$. Then the boundary in $H^{2r}(E_X,\F_2)$
of the product $b[S^2Z-Z]$
in $H^{2r+1}(S^2X-X,\F_2)$ is
$$\begin{cases}
e^{a}u+e^{a-1}\s^2u+\cdots+\s^{2a}u &\text{if }r=2a,\\
e^a\s^1u+e^{a-1}\s^3u+\cdots+\s^{2a+1}u&\text{if }r=2a+1.
\end{cases}$$
\end{lemma}

\begin{proof}
We can think of $S^2X-X$ as the interior of a real manifold with
boundary, where the boundary is the real projective bundle
$P_{\R}(T^*X)$. Then the element $b$ in $H^1(S^2X-X,\F_2)$
restricts to the class $b=w_1(O(1)_{\R})$ on $P_{\R}(T^*X)$.

As in the proof of Lemma \ref{boundarymfd}, the boundary in $H^{2r}(E_X,\F_2)$
of the product $b[S^2Z-Z]$ is the pushforward of the cohomology class
$b$ on $P_{\R}(T^*Z)$ via the map $P_{\R}(T^*Z)\inj P_{\C}(T^*X)$.
We can factor that map as $P_{\R}(T^*Z)\inj P_{\R}(T^*X)\surj P_{\C}(T^*X)$,
and the class $b$ is pulled back from $P_{\R}(TX)$. So the pushforward
of $b$ to $P_{\R}(T^*X)$ is $b$ times the class of $P_{\R}(T^*Z)$ on
$P_{\R}(T^*X)$, as computed in the proof of Lemma \ref{boundarymfd}.
Thus the pushforward of $b$ to $P_{\R}(T^*X)$
is $b(b^ru+b^{r-1}\s^1u+\cdots+\s^ru)=
b^{r+1}u+b^r\s^1u+\cdots+b\, \s^ru$.

It remains to push this class forward via the $S^1$-bundle
$h\colon P_{\R}(T^*X)\arrow P_{\C}(T^*X)$. We recall from the proof
of Lemma \ref{boundarymfd} that $h^*e=b^2$, $h_*1=0$,
and $h_*b=1$. Thus, for $Z$ of even codimension
$r=2a$, the boundary of $b[S^2Z-Z]$ in $H^{4a}(E_X,\F_2)$ is
$$h_*(b^{2a+1}u+b^{2a}\s^1u+\cdots+b\, \s^{2a}u) =
e^{a}u+e^{a-1}\s^2u+\cdots+\s^{2a}u.$$
For $Z$ of codimension $r=2a+1$, the boundary of
$b[S^2Z-Z]$ in $H^{4a+2}(E_X,\F_2)$ is
$$h_*(b^{2a+2}u+b^{2a+1}\s^1u+\cdots+b\, \s^{2a+1}u)=
e^a\s^1u+e^{a-1}\s^3u+\cdots+\s^{2a+1}u.$$
\end{proof}

Next, we prove the same formulas for any $\F_2$-cohomology class
on $X$, not necessarily the class of a submanifold. We can view
any cohomology class on a manifold as the class of a pseudomanifold,
that is, a closed piecewise linear subspace that is a manifold outside
a closed subset of real codimension at least 2.

\begin{lemma}
\label{boundary}
Let $X$ be a complex manifold, and let $u$ be an element
of $H^r(X,\F_2)$ for some $r$. Consider $u$ as the class
of a closed pseudomanifold
$Z$ in $X$. Then the boundary in $H^{2r-1}(E_X,\F_2)$
of the class $[S^2Z-Z]$
in $H^{2r}(S^2X-X,\F_2)$ is
$$\begin{cases}
e^{a-1}\s^1u+e^{a-2}\s^3u+\cdots+\s^{2a-1}u& \text{if }r=2a,\\
e^au+e^{a-1}\s^2u+\cdots+\s^{2a}u &\text{if }r=2a+1.
\end{cases}$$
Also, the boundary in $H^{2r}(E_X,\F_2)$
of the product $b[S^2Z-Z]$
in $H^{2r+1}(S^2X-X,\F_2)$ is
$$\begin{cases}
e^{a}u+e^{a-1}\s^2u+\cdots+\s^{2a}u &\text{if }r=2a,\\
e^a\s^1u+e^{a-1}\s^3u+\cdots+\s^{2a+1}u &\text{if }r=2a+1.
\end{cases}$$
\end{lemma}

\begin{proof}
By Thom, the $\F_2$-homology of any space $X$ is generated
by classes of closed (unoriented) $C^{\infty}$ manifolds $Z$
with continuous maps
$Z\arrow X$ \cite[Th\'eor\`eme III.2]{Thomcobordism}.
When $X$ is a manifold, Thom
also showed that $H_*(X,\F_2)$ is not always generated
by submanifolds; that is, we cannot always take the maps $Z\arrow X$
to be embeddings \cite[p.~46]{Thomcobordism}. For a locally compact
space $X$, Thom's argument shows that the Borel-Moore homology of $X$
with $\F_2$ coefficients is generated by $C^{\infty}$ manifolds $Z$
with proper maps $Z\arrow X$. (The Borel-Moore homology of a locally
compact space can be defined as singular homology with locally finite
chains. For a survey, see Fulton \cite[section 19.1]{Fulton}.)

Let $X$ be a complex manifold of complex dimension $n$.
By Thom's theorem,
it suffices to prove the lemma for the class $u$ in $H^r(X,\F_2)$
of a $C^{\infty}$ manifold $Z$ of real dimension $2n-r$ with a proper
map $Z\arrow X$. The idea is that for $N$ large enough, the composition
$Z\arrow X\inj X\times \P^N$ associated to a point in complex
projective space $\P^N$
can be approximated by a proper $C^{\infty}$ embedding, by Whitney. (Namely,
it suffices that $\dim_{\R}(X\times \P^N)\geq 2\dim_{\R}(Z)+1$.)
Perturbing $Z$ in this way does not change the class of $S^2Z-Z$
in $H^*(S^2(X\times\P^N)-X\times \P^N,\F_2)$.

Let $v$ be the generator of $H^2(\P^N,\F_2)$; then $v^N$ is the class
of a point on $\P^N$, and so the class of $Z$ on $X\times \P^N$
is $uv^N$. Lemmas \ref{boundarymfd} and \ref{boundaryb} compute
the boundary of the classes $[S^2Z-Z]$ and $b[S^2Z-Z]$
in $H^*(E_{X\times \P^N},\F_2)$, whether $r$ is even or odd.
For example, suppose $r=2a$ and look at the boundary of
$b[S^2Z-Z]$; the argument is completely analogous in the other three cases.
The boundary of $b[S^2Z-Z]$ in $H^{4a+4N}(E_{X\times \P^N},\F_2)$ is
$e^{a+N}uv^N+e^{a+N-1}\s^2(uv^N)+\cdots$. Since $v^N$ is in the top-degree
cohomology group $H^{2N}(\P^N,\F_2)$, we have $\s^j(v^N)=0$ for all $j>0$.
By the Cartan formula $\s^i(xy)=\sum_{j=0}^i \s^j(x)\s^{i-j}(y)$
\cite[section 4.L]{Hatcher},
the boundary of $b[S^2Z-Z]$ in $H^{4a+4N}(E_{X\times \P^N},\F_2)$
can be rewritten as
$e^{a+N}uv^N+e^{a+N-1}\s^2(u)v^N+\cdots$.

We want to compute the boundary of $[S^2Z-Z]$ in $H^{4a}(E_X,\F_2)$.
Clearly this element pushes forward to the boundary of $[S^2Z-Z]$
in $H^{4a+4N}(E_{X\times \P^N},\F_2)$. Since $E_X$ is the complex
projective bundle $P(T^*X)$
and $E_{X\times\P^N}$ is $P(T^*(X\times \P^N))$, it is straightforward
to check that this pushforward homomorphism is injective. So to show
that the boundary of $b[S^2Z-Z]$ in $H^{4a}(E_X,\F_2)$ is
$e^au+e^{a-1}\s^1u+\cdots$ as we want, it suffices to show that
the latter element pushes forward to 
$e^{a+N}uv^N+e^{a+N-1}\s^2(u)v^N+\cdots$.

We can factor the inclusion we are considering as $P(T^*X)
\inj P(T^*(X\times \P^N))|_X\inj P(T^*(X\times \P^N))$.
Here $p\colon P(T^*X)\arrow P(T^*(X\times \P^N))|_X$
is the zero set of a transverse section
of the complex vector bundle $O(1)\otimes N_{X/X\times \P^N}
=O(1)^{\oplus N}$ over $P(T^*(X\times \P^N))|_X$. So
$p_*1$ is the top Chern
class $c_N(O(1)^{\oplus N})=e^N$. So
\begin{align*}
p_*(e^au+e^{a-1}\s^1u+\cdots)&=(p_*1)(e^au+e^{a-1}\s^1u+\cdots)\\
&= e^{a+N}u+e^{a+N-1}\s^2u+\cdots.
\end{align*}
Next, we push this class forward by $q\colon
P(T^*(X\times\P^N))|_X\inj P(T^*(X\times\P^N))$. Here
$q_*1=v^N$. So
\begin{align*}
q_*(e^{a+N}u+e^{a+N-1}\s^1u+\cdots)&=(q_*1)(e^{a+N}u+e^{a+N-1}\s^1u+\cdots)\\
&= e^{a+N}uv^N+e^{a+N-1}\s^2(u)v^N+\cdots.
\end{align*}
By the previous paragraph, this proves that the formulas we want
hold in $H^*(E_X,\F_2)$.
\end{proof}

\section{Cohomology of the configuration space, and proof
of Theorem \ref{kernel}}

Milgram, L\"offler, B\"odigheimer, Cohen, and Taylor
computed the $\F_2$-homology of
the configuration space $B(X,a)$ of subsets of $X$
of order $a$ in terms of the $\F_2$-homology of $X$
and the dimension of $X$, for any compact manifold $X$
(possibly with boundary)
and any natural number $a$ \cite{ML,BCT}. Since we need explicit
generators for the cohomology of $B(X,2)=S^2X-X$,
we compute this cohomology directly for $X$ a closed manifold in Theorem
\ref{config}, not relying on their work. It would
be interesting to compute the ring structure on the $\F_2$-cohomology
of the configuration spaces $B(X,a)$
for manifolds $X$.

As a tool, we use the calculation of the homology
of symmetric products by Nakaoka and Milgram \cite{Milgram},
as follows. We need a statement (unlike Theorem \ref{symmtorsionfree})
which does not require $X$ to have torsion-free
integral cohomology. Let $f\colon X\times X\arrow S^2X$
be the obvious map.

\begin{theorem}
\label{symmprod}
Let $X$ be the complement of a closed subcomplex in a finite CW-complex.
Let $u_0,\ldots,u_s$ be a basis for $H_*^{BM}(X,\F_2)$ over $\F_2$.
Then $H_*^{BM}(S^2X,\F_2)$ has a basis consisting of the element
$f_*(u_i\otimes u_j)$
in degree $|u_i|+|u_j|$ for each $i<j$,
one element in each degree
$$|u_i|+2,|u_i|+3,\ldots,2|u_i|$$
for each i with $|u_i|>0$, and one element in degree 0
for each $i$ with $|u_i|=0$. 
\end{theorem}

\begin{proof}
First suppose that $X$ is a finite CW-complex, so that Borel-Moore
homology coincides with homology in the usual sense.
Dold showed that the $\F_p$-homology of a symmetric product
$S^aX$ (as a graded vector space) is determined by the $\F_p$-homology
of $X$ \cite[Theorem 7.2]{Dold}. So to compute the $\F_2$-homology of $S^2X$,
it suffices to compute this when $X$ is a wedge (one-point union)
of spheres. That easily reduces to the case of a single sphere $X$, where
the calculation of $H_*(S^2X,\F_2)$
was made by Nakaoka. For any finite CW-complex $X$,
the identification of some of the generators
as pushforwards $f_*(u_i\otimes u_j)$ is part of Milgram's
calculation of the ``Pontrjagin product'' on symmetric products,
the action on homology
of the natural maps $S^aX\times S^bX\arrow S^{a+b}X$
\cite[Theorem 5.2]{Milgram}.

More generally, let $X=Y-Z$ for some finite CW-complex $Y$
and closed subcomplex $Z$. Then $X$ is the complement of a point $p_0$
in the quotient space $Y/Z$, which is a finite CW-complex. So
$S^2X=S^2(Y/Z)/(Y/Z)$, where the inclusion $Y/Z\arrow S^2(Y/Z)$
is the map $x\mapsto x+p_0$. Steenrod showed that this inclusion
induces an injection on $\F_2$-homology
\cite[Theorem 2]{Dold}. So the exact sequence
$$H_j(Y/Z)\arrow H_jS^2(Y/Z)\arrow H_j^{BM}S^2X\arrow H_{j-1}(Y/Z)$$
with $\F_2$ coefficients
determines the Borel-Moore homology of $S^2X$ with $\F_2$ coefficients
from the results of Nakaoka and Milgram.
\end{proof} 

Using that, we give an explicit basis for the cohomology of $S^2X-X$.

\begin{theorem}
\label{config}
Let $X$ be a $C^{\infty}$ manifold of real dimension
$m$. Assume that $X$ can be compactified by a finite CW-complex
with complement a closed subcomplex; this is automatic
for $X$ compact. Let $v_0,\ldots,v_s$
be a basis for $H^*(X,\F_2)$, and let $Z_i$ be a closed pseudomanifold
in $X$ that represents the class $v_i$. Let $b$
in $H^1(S^2X-X,\F_2)$ be the class of the double cover
$g\colon X\times X-X\arrow S^2X-X$. Then a basis
for $H^*(S^2X-X,\F_2)$ is given by the elements
$g_*(v_i\otimes v_j)$ in degree $|v_i|+|v_j|$ for $i<j$,
together with the elements
$b^j[S^2Z_i-Z_i]$ in degree $2|v_i|+j$
for all $i$ and all $0\leq j\leq m-1-|v_i|$.
\end{theorem}

\begin{proof}
We compute the $\F_2$-Betti numbers of $S^2X-X$ by reducing
to the better-understood homology of symmetric products.
Namely, we have an
exact sequence of Borel-Moore homology
groups with $\F_2$ coefficients:
$$\arrow H_i^{BM}X\arrow H_i^{BM}S^2X
\arrow H_i^{\BM}(S^2X-X)\arrow H_{i-1}^{BM}X\arrow.$$
By Poincar\'e duality, $H_i^{\BM}(S^2X-X,\F_2)\cong H^{2m-i}(S^2X-X,\F_2)$;
so the $\F_2$-Betti numbers of $S^2X-X$ are determined by the pushforward
homomorphism associated to the diagonal inclusion
$\Delta\colon X\arrow S^2X$.

In fact, this homomorphism is zero in positive degrees. To see this,
note that it suffices to prove this for $X$ a finite CW-complex,
by the proof of Theorem \ref{symmprod}. We can also assume
that $X$ is connected.
Fix a base point $p_0$ in $X$. This determines a sequence of inclusions
$$X\arrow S^2X\arrow S^3X\arrow \cdots$$
given by adding the point $p_0$. (Do not confuse this map $X\arrow S^2X$,
$x\mapsto x+p_0$, with the diagonal inclusion, $x\mapsto 2x$.)
The colimit of this sequence is called
the {\it infinite symmetric product }$S^{\infty}X$. It can
be viewed as a topological commutative monoid, with the homotopy type
of the product of Eilenberg-MacLane spaces $\prod_{j>0} K(H_j(X,\Z),j)$,
by Dold and Thom \cite[section 4.K]{Hatcher}. This product decomposition
is compatible with the addition on $S^{\infty}X$, up to homotopy.
Moreover, by Steenrod, all the maps
$X\arrow S^2X\arrow \cdots\arrow S^{\infty}X$ given by adding $p_0$
give injections on $\F_2$-homology \cite[Theorem 2]{Dold}.

The Dold-Thom theorem implies that $R:=H^*(S^{\infty}X,\F_2)$
is a primitively generated Hopf algebra, with generators given by applying
Steenrod operations to generators of $H^j(K(H_j(X,\Z),j),\F_2)$. The
multiplication by 2 map on $S^{\infty}X$ is the composition
of the diagonal $S^{\infty}X\arrow S^{\infty}X\times S^{\infty}X$
with the addition $S^{\infty}X\times S^{\infty}X\arrow S^{\infty}X$,
and so the corresponding pullback homomorphism on $R$
is the composition
of the coproduct $R\arrow R\otimes R$ with
the product $R\otimes R\arrow R$. Pulling back the multiplication by 2
map sends a primitive class $x$ in $R$ to zero (as $x\mapsto 1\otimes x+
x\otimes 1\mapsto 2x=0$). Since $R$ is primitively generated,
the multiplication by 2 map induces zero on $R$ in positive degrees.
Equivalently, the multiplication by 2 map
induces zero on the $\F_2$-homology of $S^{\infty}X$
in positive degrees. By the commutative diagram
$$\xymatrix{
X\ar[r]^{\Delta}\ar[d] & S^2X\ar[d]\\
S^{\infty}X\ar[r]^{2} & S^{\infty}X,
}$$
the composition of the diagonal map $\Delta\colon X\arrow S^2X$
with the inclusion
$S^2X\arrow S^{\infty}X$ induces zero on $\F_2$-homology on positive degrees.
By Steenrod's theorem (above), it follows that $\Delta$ induces zero on
$\F_2$-homology in positive degrees, as we want.

We return to considering a $C^{\infty}$ manifold $X$, possibly noncompact.
We know the
Borel-Moore homology of $S^2X$ with $\F_2$ coefficients
from Theorem \ref{symmprod}. Together
with the previous paragraph, this determines the cohomology
of $S^2X-X$ with $\F_2$ coefficients, by the exact sequence
$$H_i^{BM}X\arrow H_i^{BM}S^2X\arrow H^{2m-i}(S^2X-X)\arrow H_{i-1}^{BM}X.$$
Namely, given a basis $v_0,\ldots,v_s$ for $H^*(X,\F_2)$,
$H^*(S^2X-X,\F_2)$ has a basis with one element in degree $|v_i|+|v_j|$
for all $i<j$ and one element in each degree
$$2|v_i|,2|v_i|+1,\ldots,|v_i|+m-1$$
for each $i$.

We want to show that a basis for $H^*(S^2X-X,\F_2)$ is given by
the classes
$g_*(v_i\otimes v_j)$ for $i<j$
together with the elements
$b^j[S^2Z_i-Z_i]$
for all $i$ and all $0\leq j\leq m-1-|v_i|$. Since we know the dimension
of $H^*(S^2X-X,\F_2)$, it suffices to show
that these elements are linearly independent over $\F_2$.

To see this, think of $S^2X-X$ as the interior of a manifold
with boundary, where the boundary is the real projective bundle
$P_{\R}(T^*X)$. This gives a restriction homomorphism
$$H^*(S^2X-X,\F_2)\arrow H^*(P_{\R}(T^*X),\F_2).$$
I claim that the elements $g_*(v_i\otimes v_j)$ restrict to
zero on $P_{\R}(T^*X)$. For this, think of $X\times X-X$
as the interior of a manifold with boundary,
where the boundary is the unit sphere bundle
$S_{\R}(TX)$ inside $TX$. Because the cohomology class $v_i\otimes v_j$
on $X\times X-X$ extends to $X\times X$, the restriction
of $v_i\otimes v_j$ to $S_{\R}(TX)$ extends to the unit
disc bundle $D_{\R}(TX)$, which is homotopy equivalent to $X$.
Clearly this restriction of $v_i\otimes v_j$
to $H^*D_{\R}(TX)\cong H^*X$ is $v_iv_j\in H^*X$; so the restriction
of $v_i\otimes v_j$ to $H^*S_{\R}(TX)$ is the pullback of $v_iv_j$.
So the pushforward by the double cover
$g\colon S_{\R}(TX)\arrow P_{\R}(T^*X)$ is
$g_*g^*(u_iu_j)=(g_*1)u_iu_j=0$, where $g_*1=0\in H^0(P_{\R}(T^*X),\F_2)$
since $g$ has degree $0\pmod{2}$. That is, the classes
$g_*(v_i\otimes v_j)$ restrict to zero on $P_{\R}(T^*X)$.

Next, let $u$ be the class in $H^r(X,\F_2)$ of a closed pseudomanifold
$Z$ in $X$. The restriction of $b$ in $H^1(S^2X-X,\F_2)$
to $P_{\R}(T^*X)$ is the Stiefel-Whitney class $b=w_1(O(1)_{\R})$.
By the proof of Lemma \ref{boundarymfd},
the restriction of $[S^2Z-Z]$ from $H^{2r}(S^2X-X,\F_2)$
to $H^{2r}(P_{\R}(T^*X),\F_2)$ is $b^ru+b^{r-1}\s^1u+\cdots$.
(To be precise, Lemma \ref{boundarymfd} proves this when
$Z$ is a closed $C^{\infty}$ submanifold of $X$, but the proof
of Lemma \ref{boundary}, replacing $X$ by a product
$X\times \P^N_{\C}$ (or $X\times S^N$) for $N$ large,
extends this to arbitrary
cohomology classes $u$.) So the element $b^j[S^2Z-Z]$
for $0\leq j\leq m-1-r$ restricts to
$b^j(b^ru+b^{r-1}\s^1u+\cdots)$ on $P_{\R}(T^*X)$.
Since $H^*(P_{\R}(T^*X),\F_2)$
is a free module over $H^*(X,\F_2)$ with basis
$1,b,\ldots,b^{m-1}$, we read off that the elements $b^j(S^2Z_i-Z_i)$
for $v_i=[Z_i]$ running through a basis for $H^*(X,\F_2)$
and $0\leq j\leq m-1-|v_i|$ have linearly independent
restrictions to $P_{\R}(T^*X)$.

By the previous two paragraphs, to show that the given
elements are linearly independent in $H^*(S^2X-X,\F_2)$
and hence form a basis, it suffices to show that
the elements $g_*(v_i\otimes v_j)$ for $i<j$ are linearly independent
in $H^*(S^2X-X,\F_2)$. But this is clear from the exact sequence
with $\F_2$ coefficients:
$$H_a^{BM}X\arrow H_a^{BM}S^2X\arrow H^{2m-a}(S^2X-X)\arrow H_{a-1}^{BM}X.$$
Indeed, the elements $g_*(v_i\otimes v_j)$ in the cohomology of
$S^2X-X$ are the restrictions of the Borel-Moore homology classes
$f_*(v_i\otimes v_j)$ on $S^2X$, where $f\colon X\times X
\arrow S^2X$ is the obvious map. These classes in $H_*^{BM}S^2X$ are
linearly independent for $i<j$ by Theorem \ref{symmprod}.
Since the diagonal homomorphism $H_a^{BM}X\arrow H_a^{BM}S^2X$
is zero in positive
degrees, the elements $g_*(v_i\otimes v_j)$ for $i<j$ are linearly
independent in $H^*(S^2X-X,\F_2)$. Theorem \ref{config}
is proved.
\end{proof}

\begin{proof} (Theorem \ref{kernel})
By the exact sequence of $\F_2$-cohomology groups
$$H^{j+1}X^{[2]}\arrow H^{j+1}(S^2X-X)\arrow H^jE_X\arrow H^{j+2}X^{[2]},$$
the kernel of the pushforward $i_*\colon H^*E_X\arrow H^*X^{[2]}$
is equal to the image of the boundary homomorphism
from $H^*(S^2X-X,\F_2)$. Theorem \ref{config}
gives a basis for $H^*(S^2X-X,\F_2)$, and Lemma \ref{boundary}
computes the boundary of the classes $[S^2Z-Z]$
and $b[S^2Z-Z]$, for a pseudomanifold $Z$ in $X$.
That determines the boundary of all classes $b^j[S^2Z-Z]$,
since $b^2$ in $H^2(S^2X-X,\F_2)$ is the pullback of the
class $e$ in $H^2(X^{[2]},\F_2)$. This gives the elements
of $\ker(i_*)$ listed in Theorem \ref{kernel}.

It remains to show that the boundary of each remaining
basis element $g_*(v_i\otimes v_j)$ for $H^*(S^2X-X,\F_2)$
(where $i<j$) is zero.
By the exact sequence, it suffices to show that these
classes are restrictions of cohomology classes
on $X^{[2]}$. To see this, let $\widetilde{X\times X}$
be the blow-up of $X\times X$ along the diagonal,
and let $g\colon \widetilde{X\times X}
\arrow X^{[2]}$ be the obvious degree-2 map. Since $g$ is a proper map
of manifolds, it induces a pushforward homomorphism
on cohomology. Consider a cohomology class such as $v_i\otimes v_j$
in $H^*(X\times X)$ as a class on $\widetilde{X\times X}$
by pulling back. Then the class $g_*(v_i\otimes v_j)$
on $S^2X-X$ is the restriction of the cohomology class
$g_*(v_i\otimes v_j)$ on $X^{[2]}$. Theorem \ref{kernel}
is proved.
\end{proof}

\section{Torsion-free cohomology}

\begin{proof} (Theorem \ref{torsionfree})
The Adem relations among Steenrod operations imply
that $\s^1\s^{2j}=\s^{2j+1}$ on the $\F_2$-cohomology
of any space \cite[section 4.L]{Hatcher}.
Here $\s^1$ is the Bockstein on $\F_2$-cohomology. Since
we assume that $H^*(X,\Z)$ has no 2-torsion, we have $\s^1=0$ on $H^*(X,\F_2)$,
and hence all odd Steenrod operations are zero.

As a result, Theorem \ref{kernel} gives that the kernel
of the pushforward homomorphism $i_*\colon
H^*(E_X,\F_2)\arrow H^*(X^{[2]},\F_2)$ is spanned by the elements
\begin{align*}
e^j(e^au+e^{a-1}\s^2u+\cdots+\s^{2a}u)
  & \text{ for }|u|=2a,\; 0\leq j\leq n-1-a,
\text{ and}\\
e^j(e^au+e^{a-1}\s^2u+\cdots+\s^{2a}u)
  & \text{ for }|u|=2a+1,\; 0\leq j\leq n-1-a,
\end{align*}
where $u$ runs through a basis for $H^*(X,\F_2)$.
Since $H^*(E_X,\F_2)$ is a free module over $H^*(X,\F_2)$
with basis $1,e,\ldots,e^{n-1}$, the elements listed
are linearly independent over $\F_2$.

Thus we have a basis for $\ker(i_*)$. By the exact sequence
with $\F_2$ coefficients
$$H^{j+1}X^{[2]}\arrow H^{j+1}(S^2X-X)\arrow H^jE_X\arrow H^{j+2}X^{[2]},$$
we now know the $\F_2$-Betti numbers of $X^{[2]}$. Namely,
let $v_0,\ldots,v_s$ be a basis for $H^*(X,\F_2)$.
Then $H^*(X^{[2]},\F_2)$ has a basis with one element
in degree $|v_i|+|v_j|$ for each $i\leq j$ except
when $i=j$ and $|v_i|$ is odd, together with
one element in each degree 
$$|v_i|+2,|v_i|+4,\ldots,|v_i|+2n-2$$
for each $i$.

Since $H^*(X,\Z)$ has no 2-torsion, the rational cohomology of $X$
has a basis $v_0,\ldots,v_s$ in the same degrees. To show that
$H^*(X^{[2]},\Z)$ has no 2-torsion, it suffices to show
that the rational cohomology of $X^{[2]}$ has a basis
in the same degrees as the basis above for $H^*(X^{[2]},\F_2)$.
Since the Hilbert scheme
$X^{[2]}$ is the quotient by the symmetric group $S_2$
of the blow-up $\widetilde{X\times X}$ along the diagonal,
$H^*(X^{[2]},\Q)$ is the subspace of $S_2$-invariants
in $H^*(\widetilde{X\times X},\Q)$. Since $\widetilde{X\times X}$
is the blow-up of the complex manifold $X\times X$
along the closed complex submanifold $X$ of codimension $n$, we have
$$H^*(\widetilde{X\times X})=H^*(X\times X)\oplus E\cdot H^*X
\oplus\cdots\oplus E^{n-1}\cdot H^*X,$$
where $E$ denotes the class of the exceptional divisor
$E_X$ in $H^2(\widetilde{X\times X})$. Since the nontrivial
element of $S_2$ acts
on $H^*(X\times X,\Q)$ by $v_i\otimes v_j\mapsto (-1)^{|v_i||v_j|}v_j
\otimes v_i$, the $S_2$-invariants in $H^*(X\times X,\Q)$ have
a basis with one element in each degree $|v_i|+|v_j|$ for each
$i\leq j$ except when $i=j$ and $|v_i|$ is odd. The other
summands $E^j\cdot H^*(X,\Q)$ of $H^*(\widetilde{X\times X},\Q)$
are fixed by $S_2$. We conclude that $H^*(X^{[2]},\Q)
=H^*(\widetilde{X\times X},\Q)^{S_2}$ has a basis in the same
degrees as the basis above for $H^*(X^{[2]},\F_2)$. So
the integral cohomology of $X^{[2]}$ has no 2-torsion.

Finally, suppose that $H^*(X,\Z)$ has no torsion. The easy
computation of the rational cohomology of $X^{[2]}$ above
works with $\Z[1/2]$-coefficients. In particular,
the integral cohomology of $X^{[2]}$ has no odd torsion.
By the previous paragraph, it follows that the integral
cohomology of $X^{[2]}$ is torsion-free. Theorem
\ref{torsionfree} is proved.
\end{proof}

\section{Complex submanifolds}

In a special case, the formulas in Theorems \ref{kernel}
and \ref{torsionfree} have a simple geometric explanation,
and that is what led to guessing those formulas in general. Namely,
let $Y$ be a closed complex submanifold of codimension $a$
in a complex manifold $X$,
and let $u$ be the cohomology class of $Y$ in $H^{2a}(X,\F_2)$.
The Hilbert scheme $Y^{[2]}$ is
a closed complex submanifold of codimension $2a$ in $X^{[2]}$. As throughout
the paper, we omit the symbol $\pi^*$ for cohomology
classes on $X$ pulled back to $E_X$.

\begin{lemma}
\label{complex}
The restriction of the cohomology class of $Y^{[2]}$
in $H^{4a}(X^{[2]},\F_2)$ to $E_X$
is
$$e^au + e^{a-1}\s^2u + \cdots + \s^{2a}u.$$
\end{lemma}

\begin{proof}
We have an exact sequence of holomorphic vector bundles on $Y$,
$0 \arrow TY \arrow TX|_Y \arrow N_{Y/X}\arrow 0$.
The exceptional divisor $E_X$ is the complex projective bundle
$P(T^*X)\arrow X$ of
lines in $TX$. The intersection $Y^{[2]}\cap E_X$, which is transverse,
is $W := P(T^*Y)\subset P(T^*X)|_Y\subset P(T^*X)=E_X$. Write $\pi\colon
P(T^*X)|_Y\arrow Y$ for the projection. The submanifold
$W$ is the zero set of a transverse section
of the vector bundle $\Hom(O(1),\pi^*N_{Y/X})$
over $P(T^*X)|_Y$; that section is the one
associated to the subbundle $O(-1) \subset \pi^*TX|_Y$.

So the cohomology class of $W$ on $P(T^*X)|_Y$
is the top Chern class $c_a(O(1)
\otimes N_{Y/X})$.
The top Chern class of the tensor product
of a line bundle $L$ with a vector bundle
$F$ of rank $a$ is
$$c_a(L\otimes F) = (c_1L)^a + (c_1L)^{a-1}c_1F +\cdots+ c_aF.$$
The class $e$ on $X^{[2]}$ restricted to $E_X$
is $e=c_1O(-1)$. So the cohomology class of $W$ on
$P(T^*X)|_Y$ in $\F_2$-cohomology is
$e^a + e^{a-1}c_1N_{Y/X} +\cdots+ c_aN_{Y/X}$.

The Steenrod squares of the class $u = [Y]$
in $H^*(X,\F_2)$ are the pushforward
to $X$ of the Stiefel-Whitney classes of the normal bundle
$N_{Y/X}$ by the inclusion
$s\colon Y \arrow X$,
$$\s^ju = s_*w_j(N_{Y/X}),$$
by Thom \cite{Thomsteenrod}. Since $N_{Y/X}$ is a complex
vector bundle, the odd Stiefel-Whitney classes
are zero and the even Stiefel-Whitney classes are
the Chern classes in $\F_2$-cohomology:
$$w_{2i}N_{Y/X} = c_iN_{Y/X}\pmod{2}$$
\cite[Problem 14-B]{MS}.
We conclude that the class of $W = Y^{[2]}\cap E_X$
in $H^*(E_X,\F_2)$ is
$e^au + e^{a-1}\s^2u +\cdots+\s^{2a}u$.
\end{proof}

To relate this to Theorems \ref{kernel} and \ref{torsionfree},
note that the element $e$ in $H^2(E_X,\F_2)$ is in the image
of restriction $i^*$ from $X^{[2]}$. So Lemma \ref{complex}
implies that for the class $u$ in $H^{2a}(X,\F_2)$ of a
complex submanifold, the image of restriction $i^*$
contains $e^j(e^au+e^{a-1}\s^2u+\cdots +\s^{2a}u)$
for all $j\geq 0$, in particular for all $0\leq j\leq n-1-a$.

Moreover, the class $[E_X]$ in $H^2(X^{[2]},\Z)$ is equal to $2e$,
and so $i_*1=[E_X]=0$ in $H^2(X^{[2]},\F_2)$. 
So $i_*i^*y=(i_*1)y=0$ for all $y$ in $H^*(E_X,\F_2)$. Thus Lemma
\ref{complex} shows that the kernel of $i_*$ contains
$e^j(e^au+e^{a-1}\s^2u+\cdots+\s^{2a}u)$ for all
classes $u$ in $H^{2a}(X,\F_2)$ of complex submanifolds
and all $0\leq j\leq n-1-a$. This calculation suggested
the complete description of the kernel of $i_*$
in Theorems \ref{kernel} and \ref{torsionfree}.


\small \sc UCLA Mathematics Department, Box 951555, Los Angeles, CA 90095-1555

totaro@math.ucla.edu

\end{document}